\documentclass[11pt,a4paper]{article}

\usepackage[utf8]{inputenc}
\usepackage{emptypage, enumerate}
\usepackage{amsmath}
\usepackage{amsthm}
\usepackage{amsfonts}
\usepackage{amssymb}
\usepackage{mathrsfs}
\usepackage{graphicx}
\usepackage{fancyhdr}
\usepackage{wrapfig}
\usepackage{color}
\usepackage{subfigure}
\usepackage{bbm}
\usepackage{thmtools,thm-restate, hyperref}

\newtheorem{theorem}{Theorem}
\newtheorem{corollary}[theorem]{Corollary}
\newtheorem{lemma}[theorem]{Lemma}
\theoremstyle{definition}

\newtheorem{proposition}[theorem]{Proposition}

\usepackage[dvipsnames]{xcolor}
\colorlet{k}{green!10!orange!90!}

\title{\bf Recurrence and windings of two revolving random walks}

\author{Gianluca Bosi\thanks{University of Bologna, Bologna, Italy; gianluca.bosi4@unibo.it} \and Yiping Hu\thanks{University of Washington, Seattle, USA; huypken@uw.edu} \and Yuval Peres\thanks{Kent State University, Kent, OH, USA; yuval@yuvalperes.com}}
\date{\today}
\begin{document}

\maketitle

\begin{abstract}
\noindent We study the winding behavior of random walks on two oriented square lattices. One common feature of these walks is that they are bound to revolve clockwise.
We also obtain quantitative results of transience/recurrence for each walk.
\end{abstract}


\section{Introduction}
\label{sec:intro}

\begin{figure}[h!]
\centering
\subfigure[Graph $\mathbb G_1$]
{\includegraphics[width= 2.2 in]{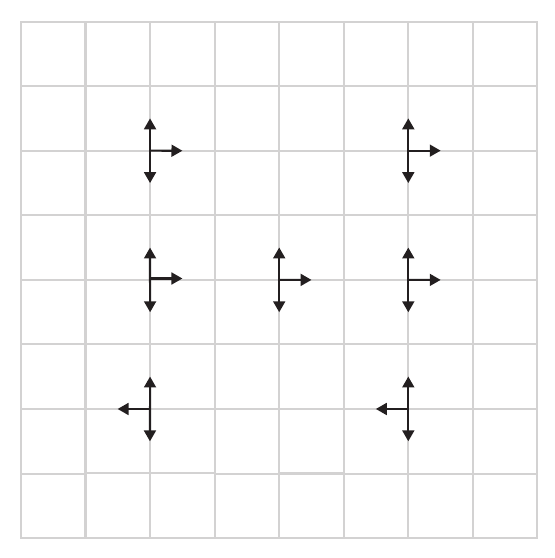}}
\subfigure[Graph $\mathbb G_2$]
{\includegraphics[width= 2.2 in]{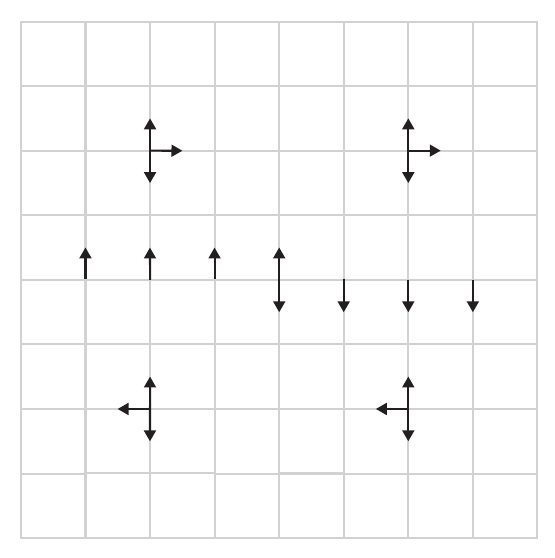}}
\caption{The graph $\mathbb G_1$ in figure (a) is transient, whereas the graph $\mathbb G_2$ in (b) is recurrent. The arrows indicate the orientation of the corresponding edges.}\label{graph}
\end{figure}
\begin{figure}[h!]
	\centering
	\subfigure[Random walk on $\mathbb G_1$]
	{\includegraphics[width= 2.3 in]{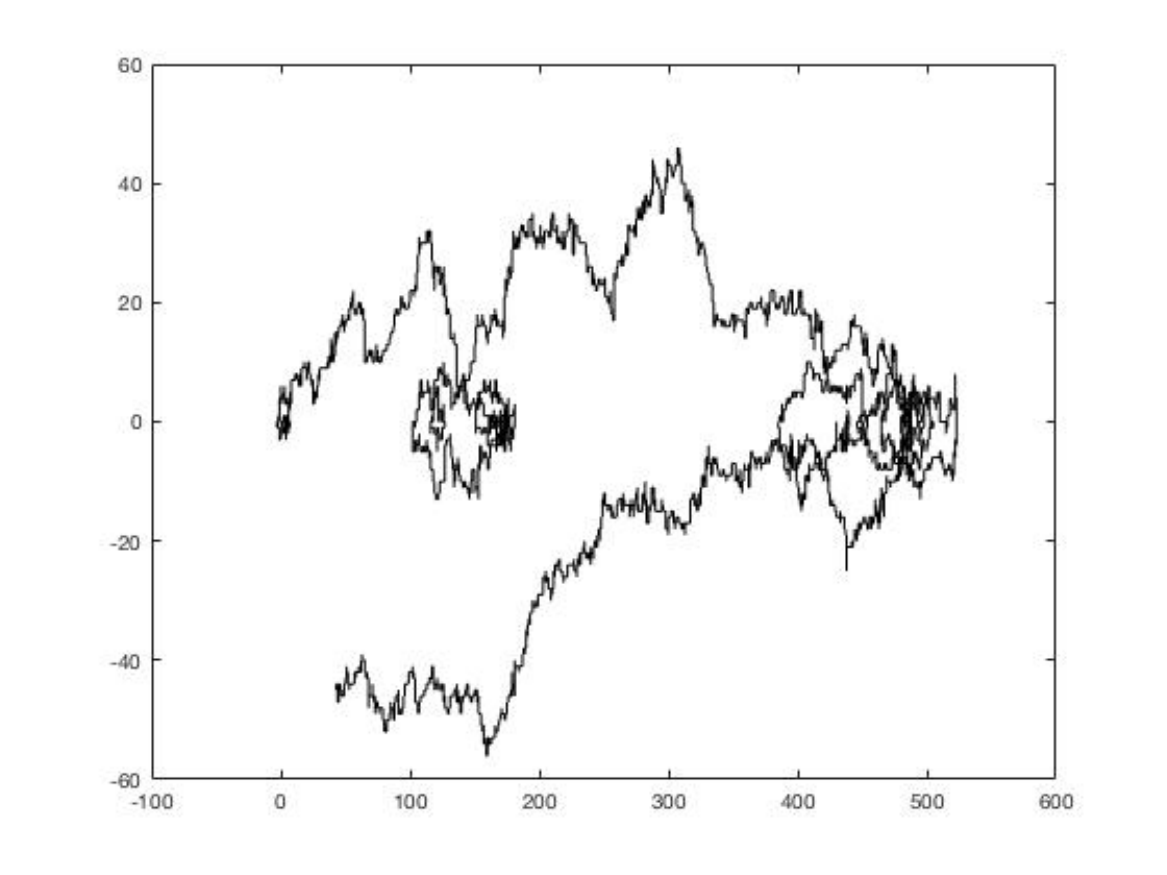}}
	\subfigure[Random walk on $\mathbb G_2$]
	{\includegraphics[width= 2.3 in]{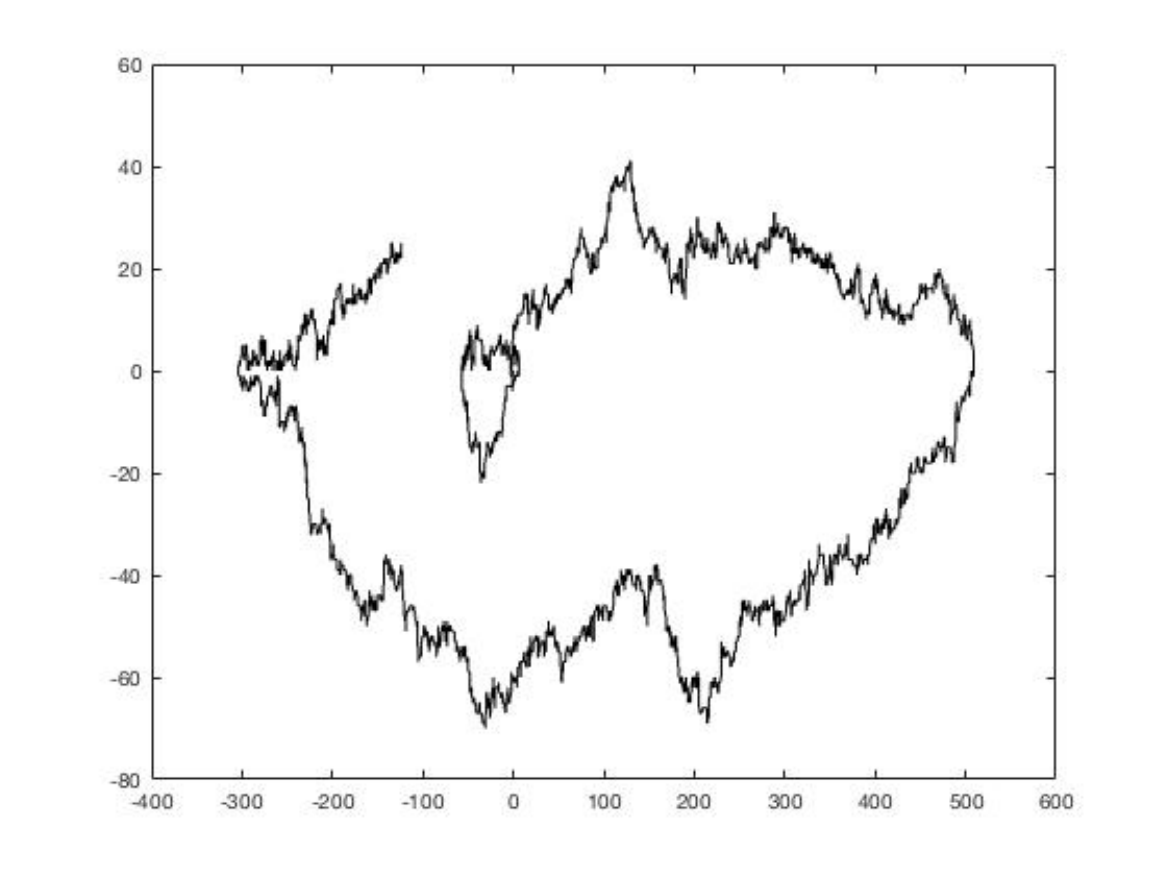}}
	\caption{Simulated trajectories of 5000 steps of the random walks on $\mathbb G_1$ and $\mathbb G_2$. Note the different scaling of the axes.}\label{graph3}
\end{figure}

Spitzer's celebrated theorem \cite{spitzer} states that the winding angle of a planar Brownian motion up to time $t$, rescaled by $\frac{1}{2}\log t$, has standard Cauchy as its limiting distribution. Since then, the winding behavior of planar processes has attracted the interest of many researchers. For the 2D simple random walk, B\'elisle \cite{Belisle} showed that its winding angle has the same scaling limit
as the big winding angle of a 2D Brownian motion, that is, the winding angle taking place outside a small ball centered at the origin. The latter is determined to be asymptotically hyperbolic secant with density $(1/2)\text{sech}(\pi u/2)$ in \cite{Messulam, Pitman}. We refer to \cite{Belisle91} for a detailed review on this topic. See also \cite{shi98, Schapira, Budd}.

In a different direction, the study of random walks on oriented lattices has intensified in the last few decades with motivations from many sources, including the Matheron-de Marsily model of transport in porous media \cite{matheron}, discretized gauge theories \cite{campanino, campanino08} and the theory of random walks in random media \cite{holmes17}. Various aspects of these models are studied (e.g. \cite{guillotin08functional, CGPS}, \cite{guillotin08, pene09, campanino14}) with many extensions \cite{devulder13, bremont20, ledger18} and connections to other models \cite{menshikov, Popov, pene20}. Except in special cases, random walks on oriented lattices are non-reversible and non-elliptic, which poses a unique set of challenges for analysis.

In this paper we study the winding behavior of the random walks on two oriented lattices $\mathbb G_1$ and $\mathbb G_2$, illustrated in Figure \ref{graph}. This is of particular interest, as both random walks are bound to revolve clockwise around the origin. After deducing the asymptotic laws of windings, we explain how these laws are closely related to more classical ones, such as the Spitzer's law. For each walk, we also derive quantitative results of transience or recurrence through our understanding of the windings.

\subsection{Models and results}

We give the precise definitions of $\mathbb G_1$ and $\mathbb G_2$. Define the directed graph $\mathbb G_1=(\mathbb{Z}^2, \mathbb E_1)$ such that a directed edge
$(v,w)=((v_1,v_2), (w_1,w_2))\in \mathbb E_1$ if and only if
$(w_1,w_2)=(v_1,v_2 \pm 1)$, or $(w_1,w_2)=(v_1+1,v_2)$ and $v_2=w_2\geq 0$, or $(w_1,w_2)=(v_1-1,v_2)$ and $v_2=w_2< 0$.
The graph $\mathbb G_2=(\mathbb V, \mathbb E_2)$ can be obtained with a slight modification of $\mathbb G_1$ by redefining only the orientations of the edges leading out from $x$-axis, that is, $((v_1,0), (w_1,w_2)) \in \mathbb E_2$ with $v_1=w_1$ and $w_2 = \pm 1$ if and only if $w_2=-1$ and $v_1=w_1>0$, or $w_2=1$ and $v_1=w_1<0$, or $w_2=\pm1$ and $v_1=w_1=0$.

Although $\mathbb G_1$ and $\mathbb G_2$ may look very similar, the random walks on them exhibit completely different behaviors. The graph $\mathbb G_1$ appeared for the first time in \cite{campanino}, where a proof of transience was given; the graph $\mathbb G_2$ was introduced later in \cite{menshikov,Popov} and the random walk on it turns out to be recurrent. The recurrence of $\mathbb G_2$ follows from
Corollary 4.8 in \cite{kemperman}, as pointed out in \cite[Prop. 7.8]{bremont}. Both random walks, considered at their successive returns to the $x$-axis, belong to the class of 1D oscillating random walks \cite{kemperman,menshikov,Popov},
with $\mathbb G_2$ critically recurrent in the class.

\vspace{.1in}

Run a simple random walk on $\mathbb{G}_1$. Let $\mathcal{N}_{\mathbb{G}_1}(n)$ be the number of windings around the origin up to the $n$-th step. See \eqref{def:NG1} for the formal definition. Our first result is a strong LLN for $\mathcal{N}_{\mathbb G_1}(n)$.

\begin{theorem} \label{LLNG1}
	\begin{equation*}
	\frac{\mathcal{N}_{\mathbb{G}_1}(n) }{\log n} \to \frac{1}{2\pi} \quad \text{a.s.}
	\end{equation*}
\end{theorem}

Note that this is in sharp contrast with the winding angle of classical 2D Brownian motion and random walks, which have nontrivial scaling limits.

In order to prove Theorem \ref{LLNG1},
we obtain a local limit theorem for the return probabilities on $\mathbb G_1$. More precisely, let $(M_i)_{i\ge0}$ be the simple random walk on $\mathbb G_1$ and let $T_n$ be the time just after the $n$-th vertical step of $M$. Write $\mathbb P_0$ for the law of $(M_i)_{i\ge0}$ starting at the origin. Then we have the following precise asymptotics:
\begin{theorem}\label{returnG1}
	$$
	\mathbb{P}_0\left(M_{T_{2n}}=(0,0)\right)\sim \frac{1}{2\sqrt{\pi}n^{3/2}}.
	$$
\end{theorem}

Theorem \ref{returnG1}, in turn, provides a new proof of the transience, see Corollary \ref{transience}.
In \cite{CGPS}, similar results as Theorem \ref{returnG1} are obtained for random walks on randomly oriented lattices.
\vspace{0.1 in}

Now consider the simple random walk on $\mathbb{G}_2$. To study its winding, we will focus on a continuous-time process $(W_t)_{t\geq 0}$  on $\mathbb R^2$, which is the scaling limit of the random walk on $\mathbb{G}_2$. Starting from the negative $x$-axis, the process $W_t$ drifts at unit speed to the right while performing a reflected Brownian motion vertically, until the first time it hits the positive $x$-axis, see Figure \ref{ladder}; then it continues analogously in the lower half plane but to the left until hitting the negative $x$-axis, and keeps alternating between two possibilities. A precise definition of $W_t$ is given in Section \ref{sec:Wt}.


\begin{figure}[h!]
	\centering
	\includegraphics[width= 5 in]{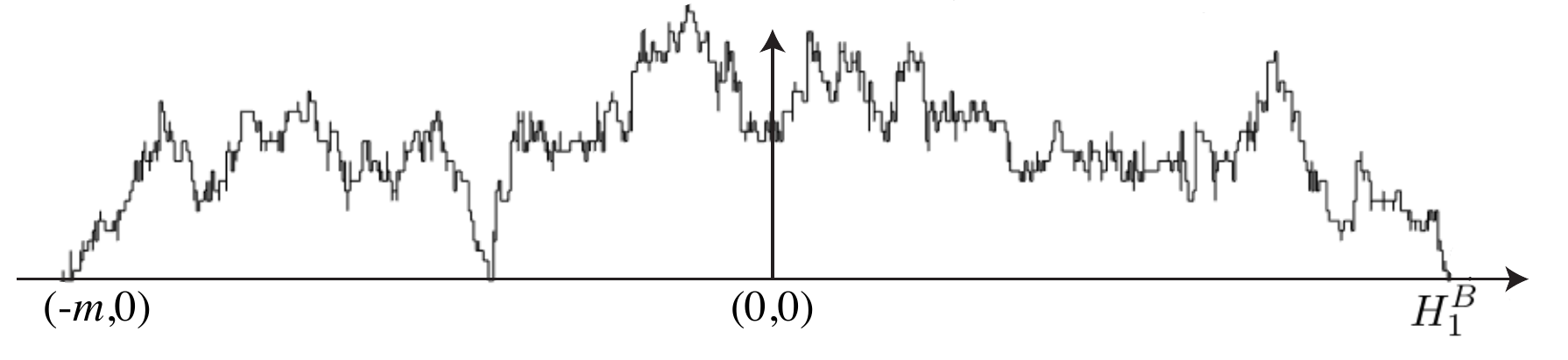}
	\caption{Illustration of the first step of the ladder height process.}\label{ladder}
\end{figure}

Let $\mathcal N_t$ be the winding number of $W_t$ around the origin up to time $t$. As shown in \cite{Belisle}, the big windings of a continuous process better capture the winding behavior of its discrete counterpart. So for $\epsilon>0$, also consider the big winding number $\mathcal N_t^b$ taking place outside a small ball of radius $\epsilon$ centered at the origin. The scaling limit in \eqref{dist2} below does not depend on the choice of $\epsilon$.


\begin{theorem}\label{windWt}
	\begin{equation} \label{dist1}
	\frac{2\pi^2 \mathcal{N}_t}{\log^2 t} \overset{d}{\implies} \rho_1
	\end{equation}
	and
	\begin{equation} \label{dist2}
	\frac{2\pi^2 \mathcal N_t^b}{\log^2 t} \overset{d}{\implies} \int_0^{\rho_1} \mathbbm{1}_{\{\beta_s>0\}}ds
	\end{equation}
	as $t \rightarrow \infty$. Here $\beta_s$ is a standard Brownian motion and $\rho_h$ represents its first hitting time at $h \in \mathbb{R}$.
\end{theorem}

Note that the limit in \eqref{dist2} has the same law as the hitting time of a reflected Brownian motion at one. So unlike the L\'evy distribution \eqref{dist1}, the distribution in \eqref{dist2} has sub-exponential tails. The comparison between \eqref{dist1} and \eqref{dist2} shows that it is the small windings near the origin that give the scaling limit of $\mathcal N_t$ its heavy tails.
Similar comments were made about the planar Brownian motion in \cite{Belisle91}.


In particular, Theorem \ref{windWt} shows that the winding and big winding numbers of $W_t$ grow faster than those of a planar Brownian motion.
The difference results from the fact that $W_t$ is only allowed to wind in the clockwise direction, whereas the planar Brownian motion chooses both directions randomly.
Surprisingly, the heuristic goes further by explaining the difference in scaling limits:
the Cauchy and hyperbolic secant distributions have the same law as the Brownian motion subordinated to an independent random time distributed as (\ref{dist1}) and (\ref{dist2}) respectively. In other words, their scaling limits are off essentially by a central limit theorem. In the same spirit, Theorem \ref{LLNG1} should be compared with the law in \cite{bertoin96}.



Our last result is about the tail of return time on $\mathbb G_2$, which quantifies its recurrence. For SRW on $\mathbb Z^2$, Dvoretzky and Erd\"os \cite{Erdos} showed that the return time to the origin has a tail of order $\Theta(1/\log k).$ By analyzing $W_t$ and exploiting the Lyapunov function methodology (see e.g. \cite{Popov}), we are able to prove a similar tail bound for $\mathbb G_2$.
Let $(X_i,Y_i)_{i \ge0}$ be the simple random walk on $\mathbb{G}_2$. Define the return time $\tau_0^+ := \min\{i \ge 1; X_i = Y_i = 0\}$.

\begin{theorem}\label{tail}
$$\lim_{k\to\infty}\frac{\log\mathbb{P}_0(\tau_0^+>k)}{\log\log k}  = -1.$$
\end{theorem}

In particular, this gives a new and self-contained proof of the recurrence of $\mathbb G_2$, see Sections \ref{sec:outline} and \ref{sec:approx}.


\subsection{Organization of the paper}

In Section \ref{sec:G1}, we shall analyze $\mathbb{G}_1$ and prove Theorems \ref{LLNG1} and \ref{returnG1}. We introduce an auxiliary process $(\mathcal G, S)$ in Section \ref{sec:G} and prove a strong LLN for its winding in Section \ref{sec:windG}. The auxiliary process $(\mathcal G, S)$ mimics the behavior of the random walk on $\mathbb G_1$ but has cleaner algebra. We come back to $\mathbb G_1$ in Section \ref{sec:returnG1}, proving Theorem \ref{returnG1} as well as the transience of $\mathbb G_1$ in Corollary \ref{transience}. Finally, in Section \ref{sec:windG1} we establish a comparison between the two processes and use the LLN for $(\mathcal G, S)$ to deduce Theorem \ref{LLNG1}.

In Section \ref{sec:G2}, we shall study $\mathbb{G}_2$ and prove Theorems \ref{windWt} and \ref{tail}. We give a precise definition of the continuous process $W_t$ in Section \ref{sec:Wt} and prove Theorem \ref{windWt} in Section \ref{sec:windWt}. In Sections \ref{sec:outline} and \ref{sec:approx}, we develop the key ingredients in the proof of Theorem \ref{tail} and prove the recurrence of $\mathbb G_2$ as an application. In Section \ref{sec:further} we prove Theorem \ref{tail}. The most technical parts of the proofs are postponed to the appendices.


\section{Random walk on $\mathbb G_1$ \label{sec:G1}}

\subsection{Auxiliary process $\mathcal G$}
\label{sec:G}

The main goal of Section \ref{sec:G1} is to prove Theorem \ref{LLNG1} for the random walk on $\mathbb G_1$. We start by introducing an analogous 2D process $(\mathcal G, S)$ with nicer algebra.

Let $S$ be a simple random walk on $\mathbb{Z}$. Recall that the graph of such a random walk is given by the path successively connecting the sequence of vertices $\{(i,S_i)\}_{i \ge 0}$ on $\mathbb{Z}^2$, with $e_i$ representing the line segment between $(i,S_i)$ and $(i+1,S_{i+1})$. We define a signed time process
\begin{equation} \label{time process}
\mathcal{G}_n := \sum_{i=0}^{n-1}\mathbf{1}_{\{ e_i \text{ is above $x$-axis} \} } - \mathbf{1}_{\{ e_i \text{ is below $x$-axis}\}}
\end{equation}
to be the difference between the time spent above and below $x$-axis. Roughly speaking, the simple random walk $S$ corresponds to the vertical movement of the random walk on $\mathbb{G}_1$, whereas
the signed time process $\mathcal{G}$ mimics the horizontal counterpart.

Let $\mathcal{N}_\mathcal{G}(n)$ be the number of windings around the origin of the two-dimensional process $(\mathcal{G},S)$ up to time $n$. We state an analogue of Theorem \ref{LLNG1} for $\mathcal{N}_\mathcal{G}(n)$. Later in Section \ref{sec:windG1}, we will establish a comparison between $\mathcal{N}_{\mathbb G_1}(n)$ and $\mathcal{N}_\mathcal{G}(n)$ and use Proposition \ref{LLNG} to prove Theorem \ref{LLNG1}.
We will prove Proposition \ref{LLNG} in Section \ref{sec:windG}.

\begin{proposition} \label{LLNG}
	\begin{equation*}
	\frac{\mathcal{N}_\mathcal{G}(n)}{\log n} \to \frac{1}{2\pi} \quad \text{a.s.}
	\end{equation*}
\end{proposition}




The following is a direct consequence of the usual Chung-Feller Theorem. See \cite{Huq} for a general introduction on the topic.

\begin{lemma}\label{returnGz}
For $z \in \{-2n, -(2n-4), \cdots, 2n-4, 2n\}$, we have
	\begin{equation}\label{uniform}
		\mathbb{P}_0\left(\mathcal G_{2n} = z \mid S_{2n} = 0\right) = \frac{1}{n+1}.
	\end{equation}
Thus for such $z$,
	$$
	\mathbb{P}_0\left((\mathcal G_{2n}, S_{2n}) = (z, 0)\right) \sim \frac{1}{\sqrt{\pi}n^{3/2}}.
	$$
The probabilities vanish for other $z$'s. \qed
\end{lemma}


\subsection{Winding of the auxiliary walk}
\label{sec:windG}

In this section we shall prove Proposition \ref{LLNG}. We will use the following definition of $\mathcal N_{\mathcal G}(2n)$:
\begin{equation*}
\mathcal{N}_\mathcal{G}	(2n) := \frac{1}{2}\sum_{i = 1}^n \mathbf{1}_{A_{2i}},
\end{equation*}
where for $i \in [1,n]$ we define $\tau_i:=\sup\{t<i;S_{2t}=0\}$ and
\begin{equation*}
A_{2i} := \{S_{2\tau_i} = S_{2i} = 0 \text{ and either } \mathcal{G}_{2\tau_i} \mathcal{G}_{2i} < 0 \text{ or } \mathcal{G}_{2\tau_i}=0 \text{ and } \mathcal{G}_{2i} > 0\}.
\end{equation*}
In words, we define $A_{2i}$ to be the event that $(\mathcal G, S)$ just completed a half winding at the $2i$-th step.
If $A_{2i}$ occurs, we say this half winding started at the $2\tau_i$-th step.
Note that since the walk is transient by Lemma \ref{returnGz}, whether we count the half windings where $\mathcal G_{2\tau_i}=0$ or $\mathcal G_{2i}=0$ wouldn't have any impact on the asymptotics in Proposition \ref{LLNG}.
Also define
\begin{equation*}
\tilde A_{2i} := \{S_{2\tau_i} = S_{2i} = 0 \text{ and }	\mathcal{G}_{2\tau_i} \mathcal{G}_{2i} \le 0 \}.
\end{equation*}

More generally, we would like to consider the law $\mathbb{P}_{2z}$ of $(\mathcal G, S)$, where the first coordinate $\mathcal G$ starts at $\mathcal G_0 = 2z$. For $n \ge 1$, we define $\mathcal G_n$ as in \eqref{time process} such that $\mathcal{G}_n := \mathcal{G}_{n-1} \pm 1$ with the sign depending on whether the edge $e_{n-1}$ is above or below $x$-axis. We use $\mathbb{P}$ without subscript to denote $\mathbb{P}_0$.

\begin{lemma}
Fix $z \in \mathbb Z$, $i \ge 1$ and $0 \le k \le i-1$, Let $m=i-k$ and $I_k = \{-k,-k+2,\cdots,k-2,k\}$. Then
\begin{equation} \label{PAcond}
 	\mathbb P_{2z}[A_{2i} \mid S_{2i} = 0, \tau_i = k] = \frac{1}{2(k+1)} \big | (-m, m) \cap (z+I_k)\big |
 \end{equation}
and
\begin{equation}\label{PtildeAcond}
	\mathbb P_{2z}[\tilde A_{2i} \mid S_{2i} = 0, \tau_i = k] \ge \frac{1}{2(k+1)} \big | [-m, m] \cap (z+I_k) \big |.
\end{equation}
Also we have
\begin{equation} \label{Pcond}
 	\mathbb P[S_{2i} = 0, \tau_i = k] \sim \frac{1}{2\pi k^{1/2} (i-k)^{3/2}}.
 \end{equation}	
\end{lemma}

\begin{proof}
	Let $U_k$ be the discrete uniform distribution on $I_k$. Let $X$ be a Rademacher random variable independent of $U_k$.
	By \eqref{uniform} the pair $(\mathcal G_{2\tau_i}, \mathcal G_{2i})$ conditioned on the event $S_{2i} = 0, \tau_i = k$ under $\mathbb P_{2z}$ has the same law as
	\begin{equation}\label{represent}
		\left(2z + 2U_k, 2z+ 2U_k + 2(i-k)X\right).
	\end{equation}
Thus the probability in \eqref{PAcond} is given by $$\mathbb P\big [|z + U_k| < i - k \text{ and } X \text{ has the correct sign}\big].$$	
This proves \eqref{PAcond}. Equation \eqref{PtildeAcond} can be proved similarly.

For \eqref{Pcond}, we simply use the Markov property and Chung-Feller Theorem.
\end{proof}

\begin{lemma}
For $i \ge 1$ and $z \in \mathbb Z$, we have
\begin{equation} \label{PA}
	\mathbb P(A_{2i}), \mathbb P(\tilde A_{2i}) \sim \frac{1}{\pi i}
\end{equation}
and
\begin{equation}\label{optimal}
	\mathbb{P}_{2z}(A_{2i}) \le \mathbb{P}_0(\tilde{A}_{2i}).
\end{equation}
When $|z| \ge i$, we have
\begin{equation}\label{trivial}
	\mathbb{P}_{2z}(A_{2i}) = 0.
\end{equation}
\end{lemma}




\begin{proof}
Using \eqref{PAcond} and \eqref{Pcond}, we get
\begin{equation*}
\mathbb P(A_{2i}) \sim \sum_{k \le i/2} \frac{1}{4\pi k^{1/2} (i-k)^{3/2}} + \sum_{k > i/2} \frac{1}{4\pi k^{3/2} (i-k)^{1/2}} \sim \frac{1}{\pi i}.
\end{equation*}
A similar calculation also works for $\mathbb P(\tilde A_{2i})$. This proves \eqref{PA}.

The inequality \eqref{optimal} follows from \eqref{PAcond}, \eqref{PtildeAcond} and the elementary fact that $$\big{|}(-m,m) \cap (z + I_k) \big{|} \le
	\big{|}[-m,m] \cap I_k \big{|}.$$ Note that the inequality in the above display would fail due to parity issue if we replaced $[-m, m]$ on the right-hand side by $(-m, m)$.

When $|z| \ge i$, we have $(-m, m) \cap (z + I_k) = \emptyset$, so $\mathbb P_{2z}(A_{2i}) = 0$.
\end{proof}

We also need the following estimates, which say that a half winding starts or completes close to the origin with small probability.

\begin{lemma} \label{smallG}
	For $i \ge 1$ and $\ell \in \mathbb N$ such that $i > 3\ell$,
	\begin{equation}\label{complete small}
		\mathbb{P}\big[ A_{2i}, |\mathcal G_{2i}|<2\ell \big] = \mathcal O\Big( \sqrt{ \frac{\ell}{i^3} }\Big).
	\end{equation}
For $0 \le \gamma \le 1$,
	\begin{equation}\label{start small}
		\mathbb{P}\big[ A_{2i}, |\mathcal G_{2\tau_i}|<2 \tau_i^\gamma \big] = \mathcal O\Big{(} \frac{1}{i^{3/2 - \gamma/2}} \Big{)}.
	\end{equation}
\end{lemma}

\begin{proof}
With the representation in \eqref{represent}, the conditional probability
\begin{equation*}
	\mathbb{P}\big[ A_{2i}, |\mathcal G_{2i}| < 2\ell \mid S_{2i} = 0, \tau_i = k \big]
\end{equation*}
is equal to
\begin{equation}\label{complete prob}
\frac{1}{2(k+1)} \#\big\{ y \in I_k; |y| < i-k \text{ and } (i-k) - |y| < \ell \big\}.	
\end{equation}
Note that the expression in \eqref{complete prob} is bounded above by $\mathcal O((i-k)/k) \wedge \mathcal O(\ell/k)$. Moreover, if $k \le (i-\ell)/2$, then for any $y \in I_k$ we have $$(i-k) - |y| \ge i-2k \ge \ell,$$ so \eqref{complete prob} vanishes. Thus we get
	\begin{equation*}
	\mathbb{P}\big[ A_{2i}, |\mathcal G_{2i}| < 2\ell \mid S_{2i} = 0, \tau_i = k \big] \le
	\begin{cases}
		0 & 0 \le k \le (i - \ell)/2,\\
		\mathcal O(\ell/k) & (i - \ell)/2 \le k \le i - \ell ,\\
		\mathcal O((i-k)/k) & i - \ell \le k \le i-1.\\
	\end{cases}
\end{equation*}
Combining the above estimate with \eqref{Pcond} yields \eqref{complete small}.

For \eqref{start small}, the representation in \eqref{represent} gives
\begin{equation*}
	\mathbb{P}\big[ A_{2i}, |\mathcal G_{2\tau_i}|<2 \tau_i^\gamma \mid S_{2i} = 0, \tau_i = k \big] = \mathcal O(k^{\gamma-1}) \wedge \mathcal O((i-k)/k).
\end{equation*}
Similarly, combining the above estimate with \eqref{Pcond} yields the desired bound.
\end{proof}


\begin{proof}[Proof of Proposition \ref{LLNG}]

By the definition of $\mathcal N_{\mathcal G}(2n)$ and \eqref{PA} we have
\begin{equation} \label{LLNExpectedSign}
	\mathbb{E}(\mathcal{N}_\mathcal{G}(2n)) \sim \frac{1}{2\pi}\log n.
\end{equation}
Our goal is to show
\begin{equation} \label{VarNG}
\text{Var}(\mathcal{N}_\mathcal{G}(2n)) \le c \log n	
\end{equation}
for some $c>0$. If both \eqref{LLNExpectedSign} and \eqref{VarNG} are true, then a Borel-Cantelli argument along the subsequence $\exp(k^{1+\epsilon})$ would imply the desired strong LLN, thanks to the monotonicity of $\mathcal N_\mathcal{G}(2n)$ in $n$.

To prove (\ref{VarNG}), it suffices to bound $\sum_{i = 1}^n \text{Var}(A_{2i})$ and the cross terms $\sum_{i=1}^n \sum_{j=1}^{i-1} \text{Cov}(A_{2j},A_{2i})$. The former is $\mathcal O(\log n)$ due to \eqref{PA}. For the cross terms, we consider two cases. Let $\alpha =2/3$.

When $1 \le j \le \alpha i$, by the Markov property, \eqref{optimal} and \eqref{PA} we get
\begin{align*}
\text{Cov}(A_{2j}, A_{2i}) &=  \mathbb{P}(A_{2j}\cap A_{2i}) - \mathbb{P}(A_{2j}) \mathbb{P}(A_{2i}) \\
&\le  \mathbb{P}(A_{2j}) \mathbb{P}(\tilde{A}_{2(i-j)}) - \mathbb{P}(A_{2j}) \mathbb{P}(A_{2i}) \\
&\sim \mathcal O \bigg{(} \frac{1}{j} \Big{(} \frac{1}{i-j} - \frac{1}{i} \Big{)} \bigg{)} = \mathcal O(1/i^2).
\end{align*}
When $\alpha i < j < i$, the above argument does not give us the desired bound. Instead we use \eqref{trivial} and \eqref{complete small} to get
\begin{align*}
\mathbb{P}(A_{2j} \cap A_{2i})
&= \sum_{|z| < i-j} \mathbb{P}_{2z}(A_{2(i-j)}) \, \mathbb{P}[\mathcal{G}_{2j} = 2z \,|\, A_{2j}] \, \mathbb{P}(A_{2j})\\
&\le \mathbb{P}(\tilde A_{2(i-j)}) \, \mathbb{P}\big[A_{2j}, |\mathcal{G}_{2j}| < 2(i-j) \big] \\
&\le \mathcal O\bigg{(} \frac{1}{i-j} \cdot  \sqrt{ \frac{i-j}{j^3} }\bigg{)} = \mathcal O\bigg{(} \frac{1}{ i^{3/2} (i-j)^{1/2} } \bigg{)}.
\end{align*}
Summing over $j$ in both cases shows that $\sum_{j=1}^{i-1} \text{Cov}(A_{2i},A_{2j})$ is of order $\mathcal O(1/i)$, so the sum of all cross terms is also $\mathcal O(\log n)$. This completes the proof.
\end{proof}

\subsection{Rate of decay of return probabilities}
\label{sec:returnG1}
In this section we start to treat the random walk on $\mathbb G_1$ and prove Theorem \ref{returnG1} and Corollary \ref{transience}.

\begin{proof}[Proof of Theorem \ref{returnG1}]
Recall that $(M_i)_{i\ge0}$ is the simple random walk on $\mathbb G_1$ and $T_n$ is the time just after the $n$-th vertical step of $M$. Consider the subordinated process
\begin{equation*}\label{decomp}
M_{T_n}=(\Xi_n, S_n),
\end{equation*}
where $S$ is the simple random walk on $\mathbb Z$, $\Xi_n:=\sum_{i=0}^{n-1} \xi_i$ and $\xi_i$ is the signed number of horizontal steps that $M$ takes between the $i-$th and the $i+1$-th vertical step. Note that $|\xi_i|$ is a geometric random variable with parameter $p=2/3$ and $\text{sgn}(\xi_i)$ determined by $\text{sgn}(S_i)$.

Define
$$\mathcal L^+_{2n}:=|\{0 \leq j <2n; S_j\geq 0\}|.$$
One can show that $\mathbb{P}_0(\mathcal L_{2n}^+=k| S_{2n}=0) \sim \frac{1}{2n}$ for $o(n) \le k \le 2n-1$ by decomposing with respect to the first time that $S$ enters the negative axis and using the generalized Chung-Feller Theorem 2.3.1 (3) in \cite{Huq}.
Then
\begin{align}\label{eq:LLTXn}
\mathbb{P}_0(\Xi_{2n}=0, &S_{2n}=0)=\sum_{k=1}^{2n} \mathbb{P}_0(\Xi_{2n}=0, S_{2n}=0, \mathcal L^+_{2n}=k)\nonumber\\
=& \sum_{k=1}^{2n} \mathbb{P}_0(\Xi_{2n}=0\mid S_{2n}=0, \mathcal L^+_{2n}=k)\mathbb{P}_0(A_{n}^+=k| S_{2n}=0)\mathbb{P}_0(S_{2n}=0)\nonumber\\
\sim&  \frac{1}{2\sqrt{\pi} n^{3/2}} \sum_{k=1}^{2n} \mathbb{P}_0(\Xi_{2n}=0\mid S_{2n}=0, \mathcal L^+_{2n}=k)\nonumber\\
=&  \frac{1}{2\sqrt{\pi} n^{3/2}} \sum_{k=1}^{2n} \mathbb{P}(\Xi_{2n,k}=0),
\end{align}
where $\Xi_{2n,k}:=\sum_{i=0}^{k-1} g_i - \sum_{i=k}^{2n-1} g_i$ for $1 \le k \le 2n$ and 
$(g_i)_{i\geq 0}$ is a sequence of i.i.d. geometric random variables with parameter $p=2/3$ and taking values in $\{0,1,2,...\}$. Let $m_{n,k}:=\mathbb{E}(\Xi_{2n,k})=k-n$ and $s_{n}:=\sigma^2(\Xi_{2n,k})=2n\sigma^2(g_1)$.
For $0<\delta<1/2$, we split the sum in (\ref{eq:LLTXn}) into two parts
\begin{equation}\label{eq:LLTXn2}
\sum_{|k-n|\leq n^{1/2+\delta}} \mathbb{P}(\Xi_{2n,k}=0) + \sum_{ |k-n| > n^{1/2+\delta}}\mathbb{P}(\Xi_{2n,k}=0).
\end{equation}
The first term in (\ref{eq:LLTXn2}) can be estimated by means of a local limit theorem for independent (not necessarily identically distributed) random variables. By \cite[Theorem 5]{Petrov} we get
\begin{align*}
\sum_{|k-n|\leq n^{1/2+\delta}} &\mathbb{P}(\Xi_{2n,k}=0) = \sum_{|k-n|\leq n^{1/2+\delta}} \left[\overline p_n^{m_{n,k}, s_n}\left(0\right) + \mathcal{O}\left(\frac{1}{n}\right) \right]\\
=& \sum_{|j|\leq n^{1/2+\delta}} \left[ \overline p_n^{0, s_n}\left(j\right) +\mathcal{O}\left(\frac{1}{n}\right) \right] = 1+o(1)+ \mathcal{O}\left(\frac{1}{n^{1/2-\delta}}\right),
\end{align*}
where $\overline p_n^{m_{n,k}, s_n}\left(x\right) =\frac{1}{\sqrt{2\pi s_{n}}}e^{-\frac{(x-m_{n,k})^2}{2s_{n}}}$.

We bound the second term in (\ref{eq:LLTXn2}) through large deviations. For $k\geq 0$ define $\hat \Xi_{2n,k}:=\Xi_{2n,k}-m_{n,k}.$ We have
\begin{align}\label{largedev}
\mathbb{P}\left(\hat \Xi_{2n,k}\geq n^{1/2+\delta} \right)
 =& \inf_{t>0} \mathbb{P}(e^{t\hat \Xi_{2n,k}}\geq e^{tn^{1/2+\delta}}) \leq \inf_{t>0} \frac{\mathbb{E}(e^{t\hat \Xi_{2n,k}})} {e^{tn^{1/2+\delta}}} \nonumber\\
 =&\inf_{t>0}\frac{\left(\frac{2e^{-t/2}}{3-e^{t}}\right)^{k}\left(\frac{2e^{t/2}}{3-e^{-t}}\right)^{2n-k}}{e^{tn^{1/2+\delta}}}=\mathcal{O}\left(e^{-\frac{n^{2\delta}}{3}}\right),
\end{align}
since by Taylor expansion
$\left(\frac{2e^{-t/2}}{3-e^{t}}\right)^{k}\left(\frac{2e^{t/2}}{3-e^{-t}}\right)^{2n-k}= 1 +\frac{3n}{4} t^2 + \mathcal{O}(nt^3).$
Analogously we have
\begin{align}\label{largedev1}
\mathbb{P}\left(\hat \Xi_{2n,k}\leq -n^{1/2+\delta} \right)=\mathcal{O}\left(e^{-\frac{n^{2\delta}}{3}}\right).
\end{align}
Combining (\ref{largedev}) and (\ref{largedev1}), we conclude
\begin{align}
\sum_{|k-n|> n^{1/2+\delta}} \mathbb{P}(\Xi_{2n,k}=0)
=&\sum_{|k-n|> n^{1/2+\delta}}  \mathbb{P}\left(\hat \Xi_{2n,k}=- \left(k-n\right)\right)=\mathcal{O}\left(ne^{-\frac{n^{2\delta}}{3}}\right) \nonumber.
\end{align}
This completes the proof of Theorem \ref{returnG1}.
\end{proof}

\begin{corollary} \label{transience}
The random walk on graph $\mathbb G_1$ is transient.
\end{corollary}
\begin{proof}
Theorem \ref{returnG1} implies the transience of $(\Xi,S)$. Thus by the translational invariance of $\mathbb G_1$ in the horizontal direction, we may find $C>0$ such that $\sum_n \mathbb{P}_0(\Xi_n=x, S_n=0)\leq C<\infty$ for every $x\in\mathbb{Z}$.
Hence $$\sum_i \mathbb{P}_0(M_i=0)=\sum_n \sum_{x \geq0} \mathbb{P}_0(\Xi_n=-x, S_n=0)(1/3)^x  \leq C\sum_{x \geq0}(1/3)^x<\infty.$$
\end{proof}

By examining the proof of Theorem \ref{returnG1}, we are able to prove a stronger version of it for generic $z$. Notice that this is an analogue of Lemma \ref{returnGz} in the setting of $\mathbb G_1$.

\begin{theorem}\label{returnG1z}
For $\delta \in (0, 1/2]$ and $z \in \mathbb N$, we have
	$$\mathbb P_0(\Xi_{2n} = z \mid S_{2n} = 0) \begin{cases}
\sim \frac{1}{2n} & |z| < n - n^{1/2+\delta},\\
\lesssim \frac{1}{2n} & |z| \in [n - n^{1/2+\delta}, n + n^{1/2+\delta}],\\
= \mathcal O(e^{-n^{2\delta}}) & |z| > n + n^{1/2+\delta}.\\
\end{cases}
$$
\end{theorem}


\subsection{Winding of the random walk on $\mathbb G_1$}\label{sec:windG1}

In this section we shall complete the proof of Theorem \ref{LLNG1}. For $1 \le i \le n$, let $\tau_i:=\sup\{t<i;S_{2t}=0\}$ and
\begin{equation*}
B_{2i} := \{S_{2\tau_i} = S_{2i} = 0 \text{ and either } \Xi_{2\tau_i} \Xi_{2i} < 0 \text{ or } \Xi_{2\tau_i} = 0 \text{ and } \Xi_{2i} > 0\}.
\end{equation*}
We use the following definition of $\mathcal{N}_{\mathbb G_1}(t)$: for $n \ge 0$ and $T_{2n} \le t <T_{2n+2}$,
\begin{equation}
\label{def:NG1}
\mathcal{N}_{\mathbb G_1}	(t) := \frac{1}{2}\sum_{i = 1}^n \mathbf{1}_{B_{2i}}.
\end{equation}

Now consider the natural coupling between the vertical components of $(\mathcal G, S)$ and $M_{T_{\cdot}} = (\Xi, S)$.
We will establish a comparison between $\mathcal{N}_\mathcal{G}(2n)$ and $\mathcal{N}_{\mathbb{G}_1}(T_{2n})$. To this end, we define a series of random variables.
Let $$\mathcal D_s := \#\{i \ge 1; A_{2i} \text{ occurs and } \text{sgn}(\mathcal G_{2\tau_i}) \neq \text{sgn}(\Xi_{2\tau_i})\}$$
and
$$\mathcal D_c := \#\{i \ge 1; A_{2i} \text{ occurs and } \text{sgn}(\mathcal G_{2i}) \neq \text{sgn}(\Xi_{2i})\}.$$

Recall the definition of $\mathcal L_{2n}^+$ from the proof of Theorem \ref{returnG1}. Let $\mathcal L_{2n}^0 := \#\{0 \le i < 2n; S_i = 0\}$.
Note that $\mathbb E(\Xi_{2n} \mid \mathcal L_{2n}^+) = \mathcal L_{2n}^+ - n$ and \begin{equation}\label{origin}
	\mathcal G_{2n}/2 + n \le \mathcal L_{2n}^+ \le (\mathcal G_{2n}/2 + n) + \mathcal L^0_{2n}.
\end{equation}
For $1/2 < \gamma < 1$, further define
\begin{enumerate}[(i)]
\item $\mathcal N_f := \#\{n \ge 0; \big|\Xi_{2n} - (\mathcal L^+_{2n} - n)\big| \ge n^\gamma\},$ \label{fluctuate}
\item $\mathcal N_{\mathcal L} := \#\{n \ge 0; \mathcal L^0_{2n} \ge n^\gamma\};$ \label{local}
\item $\mathcal N_s := \#\{i \ge 1; A_{2i} \text{ occurs and } |\mathcal G_{2\tau_i}| < 4\tau_i^\gamma\},$
\item $\mathcal N_c := \#\{i \ge 1; A_{2i} \text{ occurs and } |\mathcal G_{2i}| < 4 i^\gamma\},$
\item $\mathcal N_b := \mathcal N_f + \mathcal N_\mathcal{L} + \mathcal N_s/2 + \mathcal N_c/2;$
\item $\mathcal N'_s := \#\{i \ge 1; B_{2i} \text{ occurs and } |\Xi_{2\tau_i}| < 2\tau_i^\gamma\},$
\item $\mathcal N'_c := \#\{i \ge 1; B_{2i} \text{ occurs and } |\Xi_{2i}| < 2 i^\gamma\},$
\item $\mathcal N'_b := \mathcal N_f + \mathcal N_\mathcal{L} + \mathcal N'_s/2 + \mathcal N'_c/2.$
\end{enumerate}


We make the following two claims.

\begin{lemma}\label{comparison}
	$\mathcal N_{\mathcal G}(2n) \le \mathcal N_{\mathbb G_1}(T_{2n}) + \mathcal N_b$ and $ \mathcal N_{\mathbb G_1}(T_{2n}) \le \mathcal N_{\mathcal G}(2n)  + \mathcal N'_b$.
\end{lemma}

\begin{lemma}\label{Nb}
	$\mathcal N_b, \mathcal N'_b < \infty$ a.s.
\end{lemma}

\begin{proof}[Proof of Theorem \ref{LLNG1}]
	By Proposition \ref{LLNG} and the above two lemmas, we get $$\lim_{n\to\infty} \frac{1}{\log n} \mathcal N_{\mathbb G_1}(T_{2n}) = \lim_{n\to\infty} \frac{1}{\log n} \mathcal N_{\mathcal G}(2n) = \frac{1}{2\pi}\quad \text{a.s.}$$
	Note that $T_{2n-n^{1/2+\delta}} \le 3n \le T_{2n+n^{1/2+\delta}}$ holds a.s. for large enough $n$ and $0 < \delta <1/2$. Then Theorem \ref{LLNG1} follows from the monotonicity of $\mathcal N_{\mathbb G_1}(n)$ in $n$, see the definition \eqref{def:NG1}.
\end{proof}




\begin{proof}[Proof of Lemmas \ref{comparison}]
If $A_{2i}$ occurs but $B_{2i}$ does not occur, then either $\text{sgn}(\mathcal G_{2\tau_i}) \neq \text{sgn}(\Xi_{2\tau_i})$ or $\text{sgn}(\mathcal G_{2i}) \neq \text{sgn}(\Xi_{2i})$.
So we have
\begin{equation}\label{differ}
	\mathcal{N}_\mathcal{G}(2n) \le \mathcal{N}_{\mathbb{G}_1}(T_{2n}) + \mathcal D_s/2 + \mathcal D_c/2.
\end{equation}


Now suppose that for some $n$, the events in \eqref{fluctuate} and \eqref{local} do not happen and $|\mathcal G_{2n}| \ge 4 n^\gamma$,
then by \eqref{origin}
\begin{align*}
	\big|\Xi_{2n} - \mathcal G_{2n}/2\big|
	\le \big|\Xi_{2n} - (\mathcal L^+_{2n} - n)\big| + \big| (\mathcal L^+_{2n} - n) - \mathcal G_{2n}/2 \big| < 2n^\gamma,
\end{align*}
which implies $\mathcal G_{2n}$ and $\Xi_{2n}$ must have the same sign. Thus we get
\begin{equation*}
\mathcal D_s \le \mathcal N_f + \mathcal N_\mathcal{L} + \mathcal N_s
\text{ and }
\mathcal D_c \le \mathcal N_f + \mathcal N_\mathcal{L} + \mathcal N_c
\end{equation*}
and by \eqref{differ} we conclude that
\begin{align*}
\mathcal{N}_\mathcal{G}(2n) &\le \mathcal{N}_{\mathbb{G}_1}(T_{2n}) + \mathcal N_f + \mathcal N_\mathcal{L} + \mathcal N_s/2 + \mathcal N_c/2\\
&:= \mathcal{N}_{\mathbb{G}_1}(T_{2n}) + \mathcal N_b.
\end{align*}
The proof of the other inequality is similar, so we omit the details.
\end{proof}

To prove Lemma \ref{Nb}, we need an analogue of Lemma \ref{smallG} for the random walk on $\mathbb G_1$.

\begin{lemma} \label{smallG1}
	For $i \ge 1$ and $\ell \in \mathbb N$ such that $i > 3\ell$,
	\begin{equation}\label{complete smallG1}
		\mathbb{P}\big[ B_{2i}, |\Xi_{2i}|< \ell \big] = \mathcal O\Big( \sqrt{ \frac{\ell}{i^3} }\Big).
	\end{equation}
For $0 \le \gamma \le 1$,
	\begin{equation}\label{start smallG1}
		\mathbb{P}\big[ B_{2i}, |\Xi_{2\tau_i}|< \tau_i^\gamma \big] = \mathcal O \Big{(} \frac{1}{i^{3/2 - \gamma/2}} \Big{)}.
	\end{equation}
\end{lemma}

\begin{proof}
We imitate the proof of Lemma \ref{smallG}. The conditional probability
\begin{equation} \label{PBXicond}
 	\mathbb P[B_{2i}, |\Xi_{2i}| < \ell \mid S_{2i} = 0, \tau_i = k].
 \end{equation}
can be rewritten as
$$\frac{1}{2}\mathbb P\big [|\Xi_{2k}| < |\Xi_{2i} - \Xi_{2k}| \text{ and } |\Xi_{2i} - \Xi_{2k}| - |\Xi_{2k}| <\ell \,\big|\, S_{2i} = 0, \tau_i = k\big].$$
Note that conditioned on $S_{2i} = 0, \tau_i = k$, the law of $\Xi_{2k}$ is independent of $\Xi_{2i} - \Xi_{2k}$ and is close to being ``uniform" in $[-k, k]$ by Theorem \ref{returnG1z}.
Moreover, the Chernoff bound implies that $\Xi_{2i} - \Xi_{2k}$ concentrates around $\text{sgn}(S_{2i-1})\cdot(i-k)$ with high probability.
Thus we obtain a similar bound on \eqref{PBXicond} as in the proof of Lemma \ref{smallG}, except that the probability in the first case is exponentially small in $i$ instead of being zero. This proves \eqref{complete smallG1}. The proof of \eqref{start smallG1} is similar.
\end{proof}

\begin{proof}[Proof of Lemma \ref{Nb}]
All $\mathcal N_f$, $\mathcal N_\mathcal{L}$, $\mathcal N_s$, $\mathcal N_c$, $\mathcal N'_s$, $\mathcal N'_c < \infty$ almost surely by the concentration bound \eqref{largedev}, Theorem 10.2 and its consequences in \cite{revesz}, estimates \eqref{start small}, \eqref{complete small}, \eqref{start smallG1} and \eqref{complete smallG1} respectively. Thus $\mathcal N_b,\mathcal N'_b < \infty$ a.s.	
\end{proof}


\vspace{0.2 in}

\section{Random walk on $\mathbb G_2$}\label{sec:G2}

\subsection{The continuous process $W_t$}\label{sec:Wt}

We give a precise definition of the continuous-time process $(W_t)_{t \ge 0}:=(W_t^{(1)}, W_t^{(2)})_{t\geq 0}$ on $\mathbb R^2$, which we briefly explained in Section \ref{sec:intro}. Let $m\in\mathbb{R}^+$ and $(B_t^R)_{t \ge 0}$ be a one-dimensional reflected Brownian motion.
Inductively, we define $W_t$ together with a sequence of stopping times $(U_n)_{n\geq 0}$. Set $U_0:=0$ and $W_0:=(-m,0)$ as the initial position. For every $n \ge 1$, let $$U_{n}:=\min\big\{t>U_{n-1}+\big|W^{(1)}_{U_{n-1}}\big|;  B_t^R=0\big\}$$
and
$$
W_t:=
\begin{cases}
\big(t-U_{2n} +  W_{U_{2n}}^{(1)},B_t^R\big)  \,\,\, &\text{ if }  t\in [U_{2n}, U_{2n+1}) \text{ for some } n \geq 0, \\
\big(-t+U_{2n+1} + W_{U_{2n+1}}^{(1)},-B_t^R\big)     \,\,\, &\text{ if }  t\in [U_{2n+1}, U_{2n+2}) \text{ for some } n \ge 0.
\end{cases}
$$

In most cases needed, it suffices to keep track of $W_t$ at these random times $U_n$.
Thus we define $H_n^B:=|W_{U_{n}}^{(1)}|$ and call this discrete-time process $(H_n^B)_{n\ge0}$ with continuous state space $\mathbb R^+$ the \textit{continuous ladder height process}. Note that the ladder height process is a Markov chain in its own right.

It is straightforward to calculate the one-step distribution of $H_n^B$. Let $Z$ be a standard normal random variable and $\rho_h$ an independent variable with a L\'evy distribution, i.e., the hitting time at $h>0$ for a standard Brownian motion started at the origin. Starting from $(-m,0)$, the process $W_t$ crosses the $y$-axis at time $m$ with $y$-coordinate distributed as $\sqrt{m}|Z|$. Then the process continues until hitting the positive $x$-axis at $\rho_{\sqrt{m}|Z|}$. Thus by the space-time scaling of Brownian motion (see e.g. \cite{Feller} Vol.2 p.170), we have
\begin{equation}
\label{H1B}
H_1^B\overset{d}{=}\rho_{\sqrt{m}|Z|}\overset{d}{=}(\sqrt{m}|Z|)^2\rho_1=mZ^2\rho_1,
\end{equation}
with $Z$ and $\rho_1$ independent of each other.
As a consequence, we may represent $H_n^B$ as the product of i.i.d. random variables $\eta_n$:
\begin{equation}
\label{HnB}
H^B_n := \eta_n H_{n-1}^B = m\prod_{i=1}^n \eta_i=\exp \left( \log m+\sum_{i=1}^n \log \eta_i\right)
\end{equation}
with $\eta_1\overset{d}{=}Z^2 \rho_1.$
Since by reflection principle $\rho_1 \overset{d}{=} 1/Z^2$ (see Cor.2.22 in \cite{Legall}), it follows that $\log \eta_1$ is symmetric and, in particular, has zero mean. 
This shows the recurrence of the ladder height process $(H_n^B)_{n \ge 0}$. Indeed we have $\liminf_{n\to\infty} H_n^B = 0$. This implies that $W_{U_n}$ is recurrent and so is the continuous process $W_t$. In Section \ref{sec:outline}, we will adapt this argument to the discrete setting
and prove the recurrence of $\mathbb G_2$.

\subsection{Scaling limits of winding numbers}
\label{sec:windWt}

In this section we shall prove Theorem \ref{windWt}. First, we give rigorous definitions of $\mathcal N_t$ and $\mathcal N_t^b$.
Let $$
\mathcal{T}_n
:=\sum_{i=0}^{2n-1} \big( H_i^B + H_{i+1}^{B}\big)
$$
be the time at which $W_t$ just completed its $n$-th winding around the origin. We define the winding number $\mathcal N_t:=n$ if $\mathcal T_n \le t < \mathcal T_{n+1}$. Also define the big winding number
$$\mathcal N_t^b := \frac{1}{2} \sum_{n=0}^{2\mathcal{N}_t -1} \mathbbm{1}_{\{H_n^B > \epsilon\}},$$ which counts one half of the half windings started outside a small neighborhood of the origin with radius $\epsilon>0$.

Recall \eqref{HnB}. Let $$
\mu_n:=\max_{0 \le j \le 2n} H_j^B= m \, \exp\left(\max_{0 \le j \le 2n}\sum_{i=1}^j \log \eta_i\right).
$$
Note that
\begin{equation} \label{mu_n}
\log \mu_n \leq \log \mathcal{T}_n \leq \log (4n) + \log \mu_n,
\end{equation}
where in the second inequality, we bound each $H_i^B$ term in the definition of $\mathcal T_n$ by $\mu_n$. Also define $$\mathcal N_t^* := \min\{n \ge 0;\log \mu_{n+1} > \log t\}.$$ Since $\sum_{i=1}^j \log \eta_i$ is the sum of i.i.d. random variables with zero mean and finite variance $\sigma^2$, by applying Donsker-type theorem on the first hitting time at one, we get
$$
\frac{2\sigma^2 \mathcal{N}_t^*}{\log^2 t} \overset{d}{\implies} \rho_1.
$$
The value of $\sigma^2$ will be determined at the end of the proof.

We claim that for $0 < \alpha < 1/2$,
$$
 \mathcal N^*_{t \big/4\log^{1/\alpha} t}
 \le \mathcal N_t \le \mathcal N^*_t
$$
for large enough $t$ a.s. This would have proved \eqref{dist1}. To show the claim, note that for $0 < \alpha < 1/2$, the anti-concentration bound $\log \mu_n \ge n^\alpha$
holds for all large $n$ a.s. by a Borel-Cantelli argument along the subsequence $n = 2^k$. Thus we have $\mathcal N_t^* \le \log^{1/\alpha}t$ for all large $t$ a.s. With this fact, the claim can be proved by a straightforward argument using \eqref{mu_n} and the definitions of $\mathcal N_t$ and $\mathcal N_t^*$.
This completes the proof of \eqref{dist1} and a similar argument works for \eqref{dist2}.

To finish the proof, we calculate that $\sigma^2 = \pi^2$. Since $\eta_1\overset{d}{=}Z^2 \rho_1$ with $Z$ and $\rho_1$ independent of each other and $\rho_1 \overset{d}{=} 1/Z^2$  by the reflection principle, we have $\sigma^2 = \text{Var}(\log \eta_1 )= 8\text{Var}(\log|Z|)$. By direct computation, the cumulant-generating function of $\log|Z|$ is given by $$K(t) := \log \mathbb E[e^{t\log|Z|}] = \log \mathbb E|Z|^t = \log\Gamma\left(\frac{t+1}{2}\right) + \frac{\log 2}{2} t - \frac{\log\pi}{2},$$
where $t>-1$ and $\Gamma(s)$ is the gamma function. Using the notation of polygamma function and its reflection formula (see e.g. 6.4.1 and 6.4.7 from \cite{ASR}), we get
\begin{align*}
\text{Var}(\log|Z|) &= K''(0) = \frac{1}{4}(\log\Gamma)''(1/2)\\
&=\frac{1}{4}\psi^{(1)}(1/2)=\frac{\pi^2}{8}.
\end{align*}
Combining these gives us $\sigma^2=\pi^2$.
\qed

\subsection{Recurrence of $\mathbb G_2$: outline of proof} \label{sec:outline}
In this and the next sections we will provide a new and self-contained proof of the recurrence of $\mathbb G_2$. Simultaneously, we will develop the key ingredients in the proof of Theorem \ref{tail}, which will be treated in Section \ref{sec:further}.

Consider the random walk $(X_i,Y_i)_{i \ge 0}$ on $\mathbb G_2$. Most of the time we assume the random walk starts at $(X_0,Y_0)=(-m,0)$ for some $m \in \mathbb Z_+$ and denote its law by $\mathbb P_m$. Sometimes we also want the random walk to start at $(X_0,Y_0)=(0,h)$ for some $h \in \mathbb Z_+$, in which case we write $\mathbb P_h$.

Following the approach in Section \ref{sec:Wt}, we define a sequence of stopping times $(\tau_n)_{n \ge 0}$ and consider the \textit{discrete ladder height process} $(H_n)_{n \ge 0}$ with state space $\mathbb N$. Precisely, let $\tau_0:=0$ and for $n \ge 1$, $$\tau_n:= \inf\{i>\tau_{n-1};Y_i=0 \text{ and } X_i X_{\tau_{n-1}} \le 0 \}.$$
Then define $H_n:= |X_{\tau_n}|.$

It is not hard to see that the process $H_n$ is a Markov chain in its own right and has the same recurrence property as the original chain $(X,Y)$. In the combinatorial setting, however, we no longer have the exact representation as in \eqref{HnB}.
Instead we resort to the more robust Lyapunov function method and consider a concave function of $\log H_1$.

In the following, we stick to the convention that $\log H_1=0$ when $H_1=0$ for simplicity.
Using the inequality $\sqrt{1+x} \le 1+\frac{1}{2}x-\frac{1}{16}x^2$ for $x\in[-1,1]$, we have on the event $\{1\le H_1 \le m^2\}$ that
\begin{align*}
\sqrt{\log H_1}=&\sqrt{\log m + \log (H_1/m)} = \sqrt{\log m}\,\sqrt{1+\frac{\log (H_1/m)}{\log m}}\\
\leq& \sqrt{\log m}\,\left\{1+\frac{\log (H_1/m)}{2\log m} - \frac{1}{16} \left[\frac{\log (H_1/m)}{\log m}\right]^2  \right\}.\\
\end{align*}
Taking expectation, we get
\begin{alignat}{2}
\mathbb{E}_m\sqrt{\log H_1} \leq& \sqrt{\log m} + \frac{\mathbb{E}_m \left( \log H_1 - \log m \right)}{2\sqrt{\log m}}
-\frac{\mathbb{E}_m \left(\log H_1 - \log m\right)^2}{16(\log m)^{3/2}} \nonumber \\
+& \frac{\mathbb{E}_m \left[\left(\log H_1 - \log m\right)^2; H_1 >m^2\right]}{16(\log m)^{3/2}}  +  \mathbb E_m\left[ \sqrt{\log H_1}; H_1 > m^2 \right] \nonumber\\
\leq&\sqrt{\log m} + \frac{\mathbb{E}_m \left( \log H_1 - \log m \right)}{2\sqrt{\log m}}
-\frac{\mathbb{E}_m \left(\log H_1 - \log m\right)^2}{16(\log m)^{3/2}}  \nonumber \\
+& 2\mathbb{E}_m \left[\log^2 H_1; H_1 >m^2\right] \nonumber \\
=:&\sqrt{\log m} + \epsilon_1(m) - \epsilon_2(m) + \epsilon_3(m). \label{key estimate}
\end{alignat}
Once we show that $\epsilon_1(m)+\epsilon_3(m) << \epsilon_2(m)$ for large enough $m$, we may apply the criterion \cite[Thm.2.5.2]{Popov} on the Lyapunov function $\sqrt{\log x}$ to conclude the recurrence of $\mathbb G_2$. It remains to establish the following bounds.

\begin{lemma}\label{epsilon}
\begin{enumerate}[(i)]
\item For small enough $\delta>0$,
\begin{equation}\label{epsilon1}
\mathbb{E}_m \left( \log H_1 - \log m \right) = \mathcal{O}\left(\frac{1}{m^{1/2-3\delta}}\right).
\end{equation}	
\item There exist constants $c_1, c_2 > 0$ such that for large enough $m$,
\begin{equation}\label{epsilon2}
	c_1 \le \mathbb{E}_m\left(\log H_1 - \log m\right)^2 \le c_2.
\end{equation}
\item For small enough $\delta>0$,
\begin{equation*}\label{epsilon3}
	\epsilon_3(m) = \mathcal{O}\left( \frac{\log^2 m}{ m^{1/2-4\delta}}\right).
\end{equation*}
\end{enumerate}
\end{lemma}



\subsection{Approximation estimates} \label{sec:approx}

We shall prove the bounds in Lemma \ref{epsilon}. In all cases, the proof goes by approximating $H_1$ by its continuous counterpart $H_1^B$, using local limit theorems and Euler-Maclaurin formulas.

We will achieve the approximation through a two-stage analysis as in \eqref{H1B}. For $n \ge 1$, let
\begin{equation*}
\sigma_n := \inf\{i>\tau_{n-1}; X_i=0\} \label{def:Vn}
\end{equation*}
and define $V_n:=|Y_{\sigma_n}|$.
For $m,h,l \in \mathbb Z_+$, let $p_{m,h} := \mathbb P_m(V_1=h)$ be the probability that the random walk starting from $(-m,0)$ hits the $y$-axis at $(0,h)$ and $q_{h,l} := \mathbb P_h(H_1=l)$ the probability that the random walk started at $(0,h)$ hits the $x$-axis at point $(l,0)$. We state two local limit theorems for $p_{m,h}$ and $q_{h,l}$. Both proofs are standard, so we postpone them to Appendix \ref{app:llt}.

\begin{lemma}\label{pmh}
	For small enough $\delta>0$,
\begin{equation*}
	p_{m,h}
	= \frac{1}{\sqrt{\pi m}}e^{-\frac{h^2}{4m}} +  \mathcal{O}\left( \frac{1}{\sqrt{m} h^2}\wedge \frac{1}{m^{3/2}} + \frac{e^{-\frac{h^2}{8m}} }{m^{1-\delta}} \right)
\end{equation*}
and
\begin{equation*}
	q_{h,l}
	= \frac{ h}{2\sqrt{ \pi} l^{3/2}} e^{-\frac{h^2}{4l}} + \mathcal{O}\left( \frac{1}{l^{3/2} h}\wedge \frac{h}{l^{5/2}} + \frac{h}{l^{2-\delta}} e^{-\frac{h^2}{8l}} \right).
\end{equation*}
\end{lemma}


Recall \eqref{H1B}. Note that $\mathbb E \log \rho_1 = -2\, \mathbb{E}(\log |Z|) = \gamma + \log 2$, where $\gamma$ represents the Euler constant. Consider the following two approximation errors:
\begin{align*} 
R_f(m) :&=\mathbb E_m(\log V_1) - \mathbb E\log (\sqrt{2m}|Z|)\\
&= \mathbb E_m(\log V_1) - (\log m)/2 + \gamma/2 
\end{align*}
and
\begin{align*} 
R_g(h) :&=  \mathbb E_h(\log H_1) - \mathbb E\log(h^2 \rho_1/2)\\ &= \mathbb E_h(\log H_1) - 2\log h - \gamma.
\end{align*}

\begin{proposition} \label{RfRg}
For small enough $\delta>0$,
$$R_f(m)
=\mathcal{O}\left(\frac{1}{m^{1/2-3\delta}}\right) \text{ and }
R_g(h)
=\mathcal{O}\left(\frac{1}{h^{1-3\delta}}\right).$$
\end{proposition}

\begin{proof}
Let $f_m(x):=\frac{\log(x)}{\sqrt{\pi m}}e^{-\frac{x^2}{4m}}$ and $g_h(x):=\log x\frac{ h}{2\sqrt{\pi} x^{3/2}} e^{-\frac{h^2}{4x}}$ be two functions defined on $\mathbb R_+$. We decompose $R_f(m)$ and $R_g(h)$ as follows:
\begin{align}\label{totalerror1}
R_f(m)&:=\sum_{h=1}^\infty p_{m,h}\log h - \int_0^{\infty} f_m(x)dx =\sum_{h=1}^{\infty} \left[p_{m,h}\log h - f_m(h) \right] + \sum_{h=m^{1/2 + \delta}}^\infty f_m(h) \nonumber\\
&+ \left( \sum_{h=1}^{m^{1/2 + \delta}} f_m(h) - \int_1^{m^{1/2 + \delta}}f_m(x)dx \right) + \left( \int_1^{m^{1/2 + \delta}} f_m(x)dx - \int_0^{\infty} f_m(x)dx \right) \nonumber\\
&=: I_1 + I_2 + I_3 + I_4
\end{align}
and
\begin{align}\label{totalerror2}
R_g(h)&:=\sum_{l=1}^\infty q_{h,l} \log l - \int_0^{\infty} g_h(x)dx =\sum_{l=1}^{\infty} \left[q_{h,l} \log l - g_h(l) \right] + \sum_{l=1}^{h^{2-\delta} } g_h(l) \nonumber\\
&+ \left( \sum_{l=h^{2-\delta} }^\infty g_h(l) - \int_{h^{2-\delta} }^\infty g_h(x)dx \right) + \left( \int_{h^{2-\delta} }^\infty g_h(x)dx - \int_0^{\infty} g_h(x)dx \right) \nonumber\\
&=: J_1 + J_2 + J_3 + J_4
\end{align}
for $\delta>0$ sufficiently small.

We deal with each term in the decomposition one by one.

\begin{enumerate}[(i)]
\item Thanks to Lemma \ref{pmh} we can estimate $I_1$:
\begin{equation*}\label{I_1}
I_1=\sum_{h=1}^{\infty} \log h \, \mathcal{O}\left( \frac{1}{\sqrt{m} h^2} + \frac{e^{-\frac{h^2}{8m}} }{m^{1-\delta}} \right)=\mathcal{O}\left(\frac{\log m}{m^{1/2-\delta}}\right).
\end{equation*}
Here for the second term of the summation, we use a uniform bound for all $h \le \sqrt{m}$ and an integral to bound the sum for $h \ge \sqrt{m}$, where the error is monotone in $h$. Applying a similar splitting at $\ell = h^2$, we get
\begin{equation*}\label{J_1}
J_1=\sum_{l=1}^\infty \log l \, \mathcal{O}\left( \frac{1}{l^{3/2} h} + \frac{h}{l^{2-\delta}} e^{-\frac{h^2}{8l}} \right) =\mathcal{O}\left(\frac{\log h}{h^{1-2\delta}}\right).
\end{equation*}
\item For $I_2$ and $J_2$, an integral bound as in (i) gives
\begin{equation*}\label{I_2}
I_2=\mathcal{O}\left(e^{-m^{2\delta}}\right) \text{ and }
J_2= \mathcal{O}\left(e^{-h^\delta}\right).
\end{equation*}
\item By applying a first-order Euler-Maclaurin approximation, we obtain that
\begin{equation*} \label{I_3}
I_3 =\mathcal{O}\left(\frac{\log m}{m^{1/2-2\delta}}\right)
\text{ and }
J_3 = \mathcal{O}\left(\frac{\log h}{ h^{2- 5\delta/2} }\right).
\end{equation*}
Full details are provided in Appendix \ref{app:EMapprox}.
\item Finally, direct computations show that:
\begin{equation*}\label{I_4}
I_4
	=\mathcal{O}\left(1/\sqrt{m}\right)
\text{ and }
J_4 
	= \mathcal{O}\left(e^{-ch^\delta}\right).
\end{equation*}
\end{enumerate}

We finish the proof of Proposition \ref{RfRg} by combining (\ref{totalerror1}) and (\ref{totalerror2}) with those estimate.
\end{proof}

\begin{proposition} \label{var}
For small enough $\delta>0$,
$$\text{Var}_m(\log V_1) - \text{Var}\left( \log(\sqrt{2m}|Z|)\right)
=\mathcal{O}\left(\frac{1}{m^{1/2-3\delta}}\right) $$
and
$$
\text{Var}_h(\log H_1) - \text{Var}\left( \log(h^2\rho_1/2)\right)
=\mathcal{O}\left(\frac{1}{h^{1-3\delta}}\right).$$
\end{proposition}

\begin{proof}
By Proposition \ref{RfRg}, it suffices to show the same estimates for\begin{equation*}
\tilde R_f(m) :=\mathbb E_m(\log^2 V_1) - \mathbb E\log^2 (\sqrt{2m}|Z|)
\end{equation*}
and
\begin{equation*}
\tilde R_g(h) := \mathbb E_h(\log^2 H_1) - \mathbb E\log^2(h^2\rho_1/2).
\end{equation*}
This can be shown by going through almost the same proof as Proposition \ref{RfRg} but changing $\log$ to $\log^2$.
\end{proof}

\begin{proof}[Proof of Lemma \ref{epsilon}]
By Markov property,
\begin{align*}
\mathbb{E}_m \log (H_1 / m)
	=& \mathbb E_m\log (V_1^2 / m) + \mathbb E_m \log (H_1 / V_1^2)\\
=& \left[2\,\mathbb{E}_m(\log V_1) - \log m\right] + \sum_{h=1}^\infty \left[ \mathbb{E}_h (\log H_1)-2\log h \right] \mathbb{P}_m (V_1=h).
\end{align*}
The above calculation, together with Proposition \ref{RfRg} and Lemma \ref{pmh}, proves \eqref{epsilon1}.

To prove \eqref{epsilon2}, by \eqref{epsilon1} it suffices to show that
$$c_1 \le \text{Var}_m(\log H_1) \le c_2$$
for some $c_1, c_2 > 0$ and sufficiently large $m$.
By Proposition \ref{RfRg} we have
\begin{align*}
	\text{Var}_m(\log H_1) &= \text{Var}_m\big( \mathbb{E}_m(\log H_1 \mid V_1) \big) + \mathbb{E}_m\big( \text{Var}_m(\log H_1 \mid V_1) \big)\\
	&= \text{Var}_m\big(2\log V_1 + \mathcal O(1/V_1^{1-3\delta})\big)+ \sum_{h=1}^\infty \text{Var}_h(\log H_1) \mathbb P_m(V_1 = h).
\end{align*}
The above decomposition, combined with Proposition \ref{var} and Lemma \ref{pmh}, proves the desired variance bound.

For the truncation error $\epsilon_3(m)$, we have by Lemma \ref{pmh}
\begin{align*}
\epsilon_3(m)
=&\sum_{l=m^2}^\infty \log^2 l \, \mathbb{P}_m(H_1=l)
=\sum_{l=m^2}^\infty \sum_{h=1}^\infty p_{m,h}q_{h,l} \log^2 l \nonumber\\
\leq& \sum_{l=m^2}^\infty \log^2 l \left[\sum_{h \le \sqrt m l^\delta}    p_{m,h}q_{h,l} + \sum_{h> \sqrt m l^\delta}  p_{m,h} \right]\nonumber\\
\leq & \sum_{l=m^2}^\infty \log^2 l \left[\sum_{h \le \sqrt m l^\delta}\mathcal{O}\left(\frac{1}{\sqrt{m}}\frac{ h}{ l^{3/2}}\right)  + \mathcal{O}\left(e^{-cl^{2\delta}}\right) \right]\nonumber\\
=& \sum_{l=m^2}^\infty \log^2 l \, \mathcal{O}\left(\frac{\sqrt{m}}{ l^{3/2-2\delta}}\right) =\mathcal{O}\left( \frac{\log^2 m}{ m^{1/2-4\delta}}\right),
\end{align*}
where for $h > \sqrt{m}l^\delta$, we apply Chernoff bounds by viewing $p_{m,h}$ as the sum of $m$ many i.i.d random variables, each of which is distributed as the convolution of geometrically many Bernoulli distributions.
\end{proof}


\subsection{Further consequence} \label{sec:further}

Finally, we shall prove Theorem \ref{tail} using the results in previous sections.
For $x>0$, let $\lambda_x := \min\{n \ge 0; H_n \le x\}$.

\begin{lemma} \label{lambda}
For any $s \in [0,1/2)$, there exist constants $x_0$ and $c$ such that for $x \ge x_0$, we have $\mathbb{E}(\lambda_x^s) \le c\, \mathbb E\log H_0^{2s}$.
\end{lemma}
\begin{proof}
Note that the key estimate \eqref{key estimate} can be done for $\log^\alpha H_1$ with any $0<\alpha<1$. Then the lemma follows from the proof of \cite[Corollary 2.7.3]{Popov}.
\end{proof}



\begin{lemma}\label{max}
For any $\epsilon>0$, there exists $x_0$ such that for $x \ge x_0$, we have $$ \mathbb P\left(\max_{n \ge 0}\, \log H_{n\wedge \lambda_x} \ge y\right) \le \frac{\mathbb E (\log H_0)^{1-\epsilon}}{y^{1-\epsilon}}$$
and if $H_0 > x$ a.s.
$$P\left(\max_{n \ge 0}\, \log H_{n\wedge \lambda_x} \ge y\right) \ge \frac{c}{y^{1+\epsilon}},$$
where the constant $c>0$ only depends on $x$.
\end{lemma}

\begin{proof}
The upper bound follows by applying \cite[Corollary 2.4.6]{Popov} to $\log^\alpha H_1$ with $0 <\alpha <1$. For the lower bound, note that a similar estimate as \eqref{key estimate} implies the process $(\log H_{n\wedge\lambda_x})^\alpha$ is a submartingale for $1<\alpha<2$ and sufficiently large $x$. Then the lower bound follows from an application of optional stopping theorem, see e.g. Example 2.4.15 in \cite{Popov}.
\end{proof}

\begin{proof}[Proof of Theorem \ref{tail}]
Let $(X_i,Y_i)_{i \ge0}$ be the simple random walk on $\mathbb{G}_2$. For any $x>0$, let $\tau_x := \min\{i \ge 0; |X_i| \le x, Y_i=0\}$.

We claim that for any $\epsilon>0$, there exist a large enough $x$ and $c_1, c_2 > 0$ such that if $|X_0|>x$ and $Y_0=0$, then $$c_1 (\log k)^{-1-\epsilon} \le \mathbb P(\tau_x > k) \le c_2\, \mathbb E(\log |X_0|)^{1-\epsilon} (\log k)^{-1+\epsilon}.$$
If the claim is true, then a straightforward argument shows that for any $\epsilon>0$, there exist $c_1,c_2>0$ such that $$c_1 (\log k)^{-1-\epsilon} \le \mathbb P_0(\tau_0^+ > k) \le c_2 (\log k)^{-1+\epsilon},$$
which proves Theorem \ref{tail}.

To prove the claim, define $\tau_x^*$ to be the number of horizontal steps taken before $\tau_x$. Since $\tau_x^*$ concentrates around $\frac{1}{3}\tau_x$, it suffices to prove the same bound for the tail of $\tau_x^*$.
Analogous to \eqref{mu_n}, we have $$\max_{n \ge 0} H_{n\wedge \lambda_x} \le \tau_x^\ast \le 2\lambda_x \cdot \max_{n \ge 0} H_{n\wedge \lambda_x}.$$
Thus by Lemmas \ref{lambda} and \ref{max}, for any $\epsilon>0$, there exists a large enough $x$ and $c_2>0$ such that
\begin{align*}
\mathbb{P}(\tau_x^\ast>k) &\le \mathbb{P}\left(\lambda_x > \log^2 k\right) + \mathbb{P}\left(\max_{n \ge 0} H_{n\wedge \lambda_x} > \frac{k}{\log^2 k}\right)\\
 &\le c_2\, \mathbb E(\log |X_0|)^{1-\epsilon} (\log k)^{-1+\epsilon}.
\end{align*}
This proves the upper bound of the claim. The proof of the lower bound is similar.
\end{proof}




\subsection*{Acknowledgement}

We would like to thank anonymous referees whose comments have greatly improved this manuscript. In particular, one referee found out that the normalizing constant in Theorem \ref{windWt} has a very simple form. Another referee pointed out the important reference \cite{bremont20} to us. We are grateful to both of them. We also thank Krzysztof Burdzy and Christopher Hoffman for useful conversations. GB is supported by a Marco Polo grant from University of Bologna. YH is partially supported by a Gloria Hewitt Fellowship and a McFarlan Fellowship from University of Washington.

\appendix
\section{Local limit theorems for $p_{m,h}$ and $q_{h,l}$} \label{app:llt}

Throughout this section we shall denote the usual one-dimensional simple random walk on $\mathbb Z$ by $S$. First, we prove the local limit theorem for $p_{m,h}$.

\vspace{.1in}

Our approach is based on the fact that conditioned on the number of vertical steps before hitting the $y$-axis, the vertical movement has the same law as $S$. To calculate the probability of $n$ vertical steps, we hope to interpret the number of vertical steps before hitting $y$-axis as the sum of $m$ many i.i.d. geometric random variables $G_{p,m} := \sum_{i=1}^m g_i$ with success probability $p=1/3$ and support in $\{0,1,2,\dots\}$. The intuition is almost correct except that on graph $\mathbb G_2$, only vertical steps are allowed at ordinate zero. For this reason, we modify the transition probability of $S$ by ignoring the origin as follows: $p(1,-1)=p(1,2)=1/2$ and $p(-1,1)=p(-1,-2)=1/2,$ and write $S'$ for the resulting random walk. We also consider a 2D modification, the random walk $(X_i',Y_i')_{i\ge0}$ on an oriented graph $\mathbb G_2'$ where all the horizontal edges are to the right and all points on $x$-axis are ignored. Precisely, $\mathbb G_2'=(\mathbb V',\mathbb E_2')$ has vertex set $\mathbb V'=\mathbb Z^2\setminus \mathbb Z \times \{0\}$, and $\mathbb E_2'$ consists of all edges leading to the nearest neighbors upward, downward and to the right. Then the intuition of geometric random variables holds for the random walk on $\mathbb G_2'$, with the caveat that the conditional law of vertical movements has the same law as $S'$. For the process $(X_i,|Y_i|)_{i\ge0}$ with $y$-coordinate taking absolute value, define $p'_{m,h}$ analogously as the probability that the random walk started at $(-m,1)$ hits the $y$-axis at point $(0,h)$ for $m,h\in \mathbb Z_+$. Then
	\begin{align*}
	&p_{m,h} = p'_{m,h}
	=\sum_{\substack{n=h}}^\infty \left(\mathbb{P}_1(S'_n=-h) + \mathbb{P}_1(S'_n=h)\right) \mathbb{P}(G_{p,m}=n)\\
	=&\sum_{\substack{n=h}}^\infty \mathbb{P}_0(S_n=-h) \mathbb{P}(G_{p,m}=n) + \sum_{\substack{n=h}}^\infty \mathbb{P}_0(S_n=h-1) \mathbb{P}(G_{p,m}=n)=: p^{(1)}_{m,h} + p^{(2)}_{m,h}.
	\end{align*}
	We will focus on $p^{(1)}_{m,h}$, as $p^{(2)}_{m,h}$ can be treated analogously.
	Letting $\delta>0$, we split the sum into two parts
	\begin{align*}
	p^{(1)}_{m,h}=& \sum_{\substack{|n-2m| \leq m^{1/2 +\delta}}} \mathbb{P}_0(S_n=h) \mathbb{P}(G_{p,m}=n) +\mathcal{O}\left[\sum_{\substack{|n-2m| > m^{1/2 +\delta}}} \mathbb{P}(G_{p,m}=n) \right],
	\end{align*}
	and notice that the second term in the above display decays exponentially fast by Chernoff bound. Then, by applying the local limit theorem (see e.g. \cite{Lawler}, p.36 \footnote{This LLT and the following ones are stated for aperiodic random walks, but it is not difficult to deduce the analogue for bipartite walks, see e.g. pp. 26-27 of the cited book.}) to $S$ we obtain
	\begin{align}
	p^{(1)}_{m,h}=&\sum_{\substack{|n-2m| \leq m^{1/2 +\delta}}} \left[ \overline p_{n}(h) + \mathcal{O}\left(\frac{1}{m^{3/2}} \right)\right]\mathbb{P}(G_{p,m}=n)+ \mathcal{O}(e^{-cm^{2\delta}}) \nonumber\\
	=& \left[ \overline p_{2m}(h) + \mathcal{O}\left(\frac{1}{m^{3/2}}+\frac{e^{-\frac{h^2}{8m}}}{m^{1-\delta} } \right)\right] \sum_{\substack{|n-2m| \leq m^{1/2 +\delta}}} \mathbb{P}(G_{p,m}=n) + \mathcal{O}(e^{-cm^{2\delta}}) \nonumber \\
	=&  \left[ \overline p_{2m}(h) + \mathcal{O}\left(\frac{1}{m^{3/2}}+\frac{e^{-\frac{h^2}{8m}}}{m^{1-\delta} } \right)\right], \nonumber
	\end{align}
	where we define $\overline p_n(h):=\frac{1}{\sqrt{2\pi n}}e^{-\frac{h^2}{2n}}$ and use the first-order approximation $\overline p_{n}(h) = \overline p_{2m}(h) + \mathcal{O}\left(\frac{e^{-\frac{h^2}{8m}}}{m^{1-\delta} } \right)$ for $|n- 2m|\leq m^{1/2 + \delta}$. We conclude by noting that the other bound with an error term $\frac{1}{\sqrt{m}h^2}$ follows from the same proof, together with a different LLT in \cite{Lawler}, eq. (2.4) on p.25.

\vspace{.1in}

Next we prove the local limit theorem for $q_{h,l}$.

Let $G_{p,n}:=\sum_{k=1}^n g_k,$ with $g_k$'s i.i.d. geometric random variables with success probability $p=2/3$ and values in $\{0,1,2...\}$. Decomposing and conditioning on the number of vertical steps $n$, we have
	\begin{align*}
	q_{h,l}=&\sum_{\substack{n=h}}^\infty \mathbb{P}_0(S_n=h; S_k>0, \forall 1\leq k\leq n) \mathbb{P}(G_{p,n}=l)\nonumber\\
	=& \sum_{\substack{n=h}}^\infty \frac{h}{n}\mathbb{P}_0(S_n=h) \mathbb{P}(G_{p,n}=l),
	\end{align*}
	by the Ballot Theorem, see e.g. \cite[Thm.4.3.2]{Durrett}.
	Now let $\delta>0$ and split the sum into two parts as follows
	\begin{align}\label{eq:split}
	\sum_{\substack{|n-2l| \leq l^{1/2 +\delta}}} \frac{h}{n}
	\mathbb{P}_0(S_{n}=h)\mathbb{P}(G_{p,n}=l)
	+& \mathcal{O}\left[\sum_{\substack{|n-2l| > l^{1/2 +\delta}}} \mathbb{P}(G_{p,n}=l) \right].
	\end{align}
	Notice that as $G_{p,n}$ has a negative binomial distribution,
	\begin{equation}
	\mathbb{P}(G_{p,n}=l) = \binom{n+l-1}{l} p^n \left(1-p\right)^{l} = \frac{n}{l}\mathbb{P}( G_{1-p,l}=n),
	\label{Lemma:YB}
	\end{equation}
	so for the second term of (\ref{eq:split}), we have
	\begin{align*}
	\sum_{\substack{|n-2l| > l^{1/2 +\delta}}}\mathbb{P}(G_{p,n}=l)
	=&\sum_{\substack{|n-2l| > l^{1/2 +\delta}}}\frac{n}{l}\mathbb{P}( G_{1-p,l}=n) \\
	\leq & \mathbb{E}\left[G_{1-p,l};|G_{1-p,l}- 2l|\geq l^{1/2 +\delta}\right]= \mathcal{O}(e^{-cl^{2\delta}}),
	\end{align*}
	for appropriate $c>0$ by the Chernoff bound. By (\ref{Lemma:YB}) again, we can rewrite the first term of (\ref{eq:split}) as
	$$
	\sum_{\substack{|n-2l| \leq l^{1/2 +\delta}}} \frac{h}{l}
	\mathbb{P}_0(S_{n}=h)\mathbb{P}(G_{1-p,l}=n)
	$$
	and apply the local limit theorems and first order approximation as before.

\section{Euler-Maclaurin approximation} \label{app:EMapprox}

In this section we will apply the Euler-Maclaurin formula to bound $I_3$ and $J_3$.

Recall that $f_m(x):=\frac{\log(x)}{\sqrt{\pi m}}e^{-\frac{x^2}{4m}}$ and $f_m'(x)=\left(\frac{1}{x} - \frac{x \log x}{2m}\right)\frac{1}{\sqrt{\pi m}}e^{-\frac{x^2}{4m}}$.
Hence, by the Euler-Maclaurin formula
\begin{align*}
I_3\leq & \sum_{k=0}^{(1/2+\delta)\log_2 m} \left[\sum_{h=2^k}^{2^{k+1}} f_m(h) - \int_{2^k}^{2^{k+1}} f_m(x)dx\right] \nonumber\\
= & \sum_{k=0}^{(1/2+\delta)\log_2 m} \left[ \frac{f_m(2^k)+f_m(2^{k+1})}{2} +  r_k \right]=\mathcal{O}\left(\frac{\log m}{m^{1/2-2\delta}}\right),
\end{align*}
where $r_k$ denotes the $k$-th error term and the last equality follows from
\begin{align*}
|r_k|\leq C2^k\max_{2^k\leq x \leq 2^{k+1}} |f'_m(x)| \leq& C2^k\max_{2^k\leq x \leq 2^{k+1}}  \left(\frac{1}{x} + \frac{x \log x}{2m}\right)\frac{1}{\sqrt{\pi m}}e^{-\frac{x^2}{4m}}\\
\leq & C 2^k \left(\frac{1}{2^k} + \frac{2^{k+1} (k+1)}{2m}\right)\frac{1}{\sqrt{\pi m}}\\
=& \mathcal{O}\left(\frac{1}{\sqrt{m}} + \frac{ 2^{2k} k}{m^{3/2}}\right).
\end{align*}

Let $g_h(x):=\log x\frac{ h}{2\sqrt{\pi} x^{3/2}} e^{-\frac{h^2}{4x}}$ and $g_h'(x)=\left(1 -\frac{3 \log x}{2} +\frac{h^2\log x}{4 x} \right)\frac{ h}{2\sqrt{ \pi} x^{5/2}} e^{-\frac{h^2}{4x}}$. By the Euler-Maclaurin formula,
\begin{align*}
J_3\leq & \sum_{k=(2-\delta)\log_2 h}^{\infty}  \left[\sum_{l=2^k}^{2^{k+1}} g_h(l) - \int_{2^k}^{2^{k+1}} g_h(x)dx\right] \leq  \sum_{k=(2-\delta)\log_2 h}^{\infty} \left[ \frac{g_h(2^k)+g_h(2^{k+1})}{2} + \tilde r_k \right] \nonumber \\
=& \sum_{k=(2-\delta)\log_2 h}^{\infty} \mathcal{O}\left(\frac{ h k}{2^{3k/2}} + \frac{h^3k}{2^{5k/2}} \right) = \mathcal{O}\left(\frac{\log h}{ h^{2-\frac{5\delta}{2}}}\right),
\end{align*}
where we use the fact that
\begin{align*}
|\tilde r_k|\leq &  C'2^k\left(1 +\frac{3  (k+1)}{2} +\frac{h^2 (k+1)}{2^{k+2} } \right)\frac{ h}{2\sqrt{ \pi} 2^{5k/2}}\\
=&\mathcal{O}\left(\frac{h k}{2^{3k/2}} + \frac{h^3k}{2^{5k/2}}\right).
\end{align*}


\bibliography{RevolvingRW}{}
\bibliographystyle{amsplain}

\end{document}